\documentclass[12pt,reqno]{amsart}
\usepackage{amsmath,amssymb}
 \makeatletter
 \evensidemargin  \oddsidemargin
 \makeatother

\newtheorem{theorem}{Theorem}
\theoremstyle{plain}

\newtheorem{corollary}{Corollary}

\newtheorem{example}{Example}

\newtheorem{lemma}{Lemma}

\newtheorem{proposition}{Proposition}
\newtheorem{remark}{Remark}

\numberwithin{equation}{section}

 \numberwithin{equation}{section}

\begin{document}

\title[Norm inequalities related to operator monotone functions]{Norm inequalities related to operator monotone functions}
\author[A. G. Ghazanfari]{ A. G. Ghazanfari}

\address{Department of Mathematics, Lorestan University, P. O. Box 465, Khoramabad, Iran.}

\email{ghazanfari.a@lu.ac.ir}

\setcounter{page}{1}


\begin{abstract}
Let $A$ be a positive definite operator on a Hilbert space $H$, and $|||.|||$ be a unitarily invariant norm on $B(H)$.
We show that if $f$ is an operator monotone function on $(0,\infty)$ and $n\in \mathbb{N}$, then $|||D^n f(A)|||\leq\|f^{(n)}(A)\|$
and $\|f^{(n)}(\cdot)\|$ is a quasi-convex function on the set of all positive definite operators in $B(H)$.
We establish some estimates of the right hand side of some Hermite-Hadamard type inequalities
in which differentiable functions are involved, and norms of the maps induced by them on the set of
self adjoint operators are convex, quasi-convex or $s$-convex.

As applications, we obtain some of bounds for $|||f(B)-f(A)|||$ in term of $|||B-A|||$. For instance,
Let $f,g$ be two operator monotone functions on $(0,\infty)$. Then, for every unitarily invariant norm $|||.|||$ and
every positive definite operators $A,B$,
\begin{align*}
&\left|\left|\left|f(A)g(A)-f(B)g(B)\right|\right|\right|\notag\\
&\leq|||B-A|||\Big[\max\left\{\|f'(A)\|,\|f'(B)\|\right\}\times\max\left\{\|g(A)\|,\|g(B)\|\right\}\notag\\
&+\max\left\{\|f(A)\|,\|f(B)\|\right\}\times \max\left\{\|g'(A)\|,\|g'(B)\|\right\}\Big].
\end{align*}
\end{abstract}

\maketitle

\section{Introduction and Preliminary}

We recall that the definitions of quasi-convex and $s$-convex functions generalize the definition of convex
function. More exactly, a function $f : [a, b] \rightarrow \mathbb{R} $ is said to be quasi-convex on $[a, b]$ if
for all $x,y\in [a,b]$ and $0\leq\lambda\leq1$,
\begin{align*}
f((1-\lambda)x+\lambda y)\leq \max\{f(x),f(y)\},
\end{align*}
and for a $s$ fixed in $(0,1]$, a function $f : (0,\infty)\rightarrow \mathbb{R}$ is said to be $s$-convex in the second sense \cite{hud}
if for all $x,y\in(0,\infty)$ and $0\leq\lambda\leq1$,
\begin{align*}
f((1-\lambda)x+\lambda y)\leq (1-\lambda)^sf(x)+\lambda^sf(y).
\end{align*}

Let $B(H)$ denote the set of all bounded linear operators on a complex Hilbert
space $H$.
An operator $A \in B(H)$ is positive definite (resp. positive semi-definite)
if $\langle Ax, x\rangle > 0$ (resp. $\langle Ax, x\rangle \geq 0 )$ holds for all non‐zero $x \in H$ . If $A$ is
positive semi‐definite, we denote $A\geq 0$.
To reach inequalities for bounded self-adjoint operators on Hilbert space, we shall use
the following monotonicity property for operator functions:\\
If $X\in B(H)$ is self adjoint with a spectrum $Sp(X)$, and $f,g$  are continuous real valued functions
on an interval containing $Sp(X)$, then
\begin{equation}\label{1.1}
f(t)\geq g(t),~t\in Sp(X)\Rightarrow ~f(X)\geq g(X).
\end{equation}
For more details about this property, the reader is referred to \cite{pec}.

 Suppose that $I$ be an interval in $\mathbb{R}$ and let
 $$\sigma(I)=\{A\in B(H) : \text{ $A$ is self adjoint and }
 Sp(A)\subseteq I\}.$$
 A real valued continuous function $f$
on an interval $I$ is said to be operator convex (operator concave)
if
\begin{align*}
f\left((1-\lambda)A+\lambda
B\right)\leq(\geq)(1-\lambda)f(A)+\lambda f(B)
\end{align*}
in the operator order in $B(H)$, for all $\lambda\in [0,1]$ and for
every bounded self adjoint operators $A$ and $B$ in $\sigma(I)$.\\
A real valued continuous function $f$
on an interval $I$ is said to be operator monotone if it is monotone with respect to
the operator order, i.e.,
\[
A\leq B \text{ implies } f(A)\leq f(B)
\]
for every bounded self adjoint operators $A$ and $B$ in $\sigma(I)$.\\
For some fundamental results on operator convex (operator concave)
and operator monotone functions, see \cite{bha, pec} and the
references therein.

Let $f$ be a real function on $(0,\infty)$, and let $f^{(n)}$ be its
$n$th derivative. Let $f$ also denote the map induced by $f$ on
positive semi-definite operators. Let $D^nf(A)$ be the $n$th order
Fr$\acute{e}$chet derivative of this map at the point $A$. For each
$A$, the derivative $D^nf(A)$ is a $n$-linear operator on the space
of all Hermitian operators. The norm of this operator is defined as
\begin{align*}
\left\|D^nf(A)\right\|=\sup\left\{\left\|D^nf(A)\left(B_1,...,B_n\right)\right\|:
\|B_1\|=...=\|B_n\|=1\right\}.
\end{align*}

Since $f^{(n)}(A)=D^nf(A)(1_H, 1_H,...,1_H)$, we have
$$\|f^{(n)}(A)\|=\|D^nf(A)(1_H, 1_H,...,1_H)\|\leq \|D^nf(A)\|.$$

Now, let $\mathcal{D}^{(n)}=\left\{f:\left\|D^nf(A)\right\|=\|f^{(n)}(A)\|~
\text{for all positive operator}~A\right\}$.

In \cite{bha3} it was shown that every operator monotone function is
in $\mathcal{D}^{(n)}$ for all $n = 1, 2, ...$. It was also shown in
\cite{bha2} that the functions $f(t) = t^n$, $n = 2, 3, ...,$ and
$f(t) = \exp (t)$ are also in $\mathcal{D}^{(1)}$. None of these are operator
monotone. In \cite{bha4} it was shown that the power function $f(t)
= t^p$ is in $\mathcal{D}^{(1)}$ if $p$ is in $(-\infty, 1]$ or in
$[2,\infty)$, but not if $p$ is in $(1,~\sqrt{2})$.

A norm $||| \cdot |||$ on $B(H)$ is called unitarily invariant
norm if $|||UAV ||| = |||A|||$ for all $A\in B(H)$ and all unitary operators $U, V \in B(H)$.
If $|||.|||$ is any unitarily invariant norm on Hermitian operators, then the corresponding norm
of the $n$-linear $D^nf(A)$ is defined as
\begin{align*}
|||D^nf(A)|||=\sup\left\{|||D^nf(A)\left(B_1,...,B_n\right)|||:
|||B_1|||=...=|||B_n|||=1\right\}.
\end{align*}
In this paper, we consider differentiable mappings which
  norm of the induced maps by them on the set of
self adjoint operators is convex, quasi-convex or $s$-convex.
We show that if $f$ is an operator monotone function on $(0,\infty)$, $A$ is a positive definite operator and $|||.|||$ a unitarily invariant norm, then
$|||D^n f(A)|||\leq\|f^{(n)}(A)\|$ for all positive integers $n$.
We also prove that $\|f^{(n)}(\cdot)\|$ is a quasi-convex function. Examples and applications for particular cases of
interest are also illustrated. Finally, two error estimates for the Simpson formulas are addressed.

\section{Main results}
\subsection{Quasi-convexity and operator monotone functions}
In some problems of approximation
theory and perturbation theory, when we deal with norms of functions, it usually is not needed that the functions have
a special property such as convexity, monotony, quasi-convexity and $s$-convexity; just having the property
for norms of the functions would be sufficient. Let $f$ be a real function on an interval $I$ in $\mathbb{R}$, throughout this paper
the function $\|f(\cdot)\|$ is the norm of the induced map by $f$ on the set of
self adjoint operators as follows:
$$\|f(\cdot)\|: \sigma(I)\rightarrow [0,\infty); A\mapsto \|f(A)\|.$$
First, we give the following three results:

(1) Let $f$ be a positive real function on an interval $I\subseteq \mathbb{R}$. If $f$ is operator convex,
the relation $0\leq f(1-t)A+tB)\leq (1-t)f(A)+tf(B)$ implies that
$\|f(1-t)A+tB)\|\leq (1-t)\|f(A)\|+t\|f(B)\|$, i.e., $||f(\cdot)||$ is convex.

(2) Let $f$ be a positive real function on $I\subseteq [0,\infty)$. if $f$ is operator monotone, then
for every positive semi-definite operator $A\in \sigma(I)$,
\begin{align*}
f(\|A\|)&=f\left(\sup\{\lambda : \lambda\in \sigma(A)\}\right)\\
&=\sup\{f(\lambda) : \lambda\in \sigma(A)\}=\|f(A)\|.
\end{align*}
Therefore
\begin{align*}
\left\|f((1-t)A+tB)\right\|&=f(\|(1-t)A+tB\|)\leq f\left((1-t)\|A\|+t\|B\|\right)\\
&\leq f\left(\max\{\|A\|,\|B\|\}\right)=\max\{f(\|A\|),f(\|B\|)\}\\
&=\max\{\|f(A)\|,\|f(B)\|\},
\end{align*}
i.e., $||f(\cdot)||$ is quasi-convex.

(3) Let $f$ be an operator monotone function on $[0,\infty)$ and $|||\cdot|||$ a unitarily invariant norm on $B(H)$. Then
 $|||f(X)|||\leq f(\|X\|)|||1_H|||$ for all positive semi-definite operators $X$, since $X\leq \|X\|.1_H$. This implies that
\begin{align*}
|||f((1-t)A+tB)|||&\leq f(\|(1-t)A+tB\|)|||1_H|||\\
&\leq\max\{\|f(A)\|,\|f(B)\|\}|||1_H|||,
\end{align*}
Hence, we get the following result.

\begin{proposition}\label{p1}
\begin{enumerate}
\item[(i)]
Let $f$ be a positive real function on $I\subseteq \mathbb{R}$. If $f$ is operator convex Then
$\|f(\cdot)\|$ is convex on $\sigma(I)$.
\item[(ii)]
Let $f$ be a positive real function on $I\subseteq [0,\infty)$. If $f$ is operator monotone. Then
$\|f(\cdot)\|$ is quasi-convex on $\sigma(I)$.

\item[(iii)]
Let $f$ be an operator monotone function on $[0,\infty)$ such that $f(0)=0$. Then $\frac{1}{|||1_H|||}|||f(\cdot)|||$
is quasi-convex on $\sigma([0,\infty))$, for every unitarily invariant norm $|||\cdot|||$ on $B(H)$.
\end{enumerate}
\end{proposition}

\begin{example}\label{e1}
\begin{enumerate}
\item[(i)]
Let $f$ be the real function defined by $f(t) = t^{r}~(r\in\mathbb{R})$ on $(0,\infty)$.
If $r\notin (0,1)$, then the map $\|f(\cdot)\|$ is convex on $\sigma((0,\infty))$. Because for $r\geq 1$, we have
\begin{align*}
\|((1-t)A+tB)^{r}\|&=\left\|(1-t)A+tB\right\|^r\leq\big((1-t)\|A\|+t\|B\|\big)^r\\
&\leq(1-t)\|A\|^r+t\|B\|^r=(1-t)\|A^r\|+t\|B^r\|.
\end{align*}
For $r\leq -1$, we obtain
\begin{align*}
\|((1-t&)A+tB)^{r}\|=\left\|((1-t)A+tB)^{-1}\right\|^{-r}\\
&\leq \left\|((1-t)A^{-1}+tB^{-1})\right\|^{-r}\leq\big((1-t)\|A^{-1}\|+t\|B^{-1}\|\big)^{-r}\\
&\leq(1-t)\|A^{-1}\|^{-r}+t\|B^{-1}\|^{-r}=(1-t)\|A^r\|+t\|B^r\|,
\end{align*}
and for $-1\leq r\leq 0$, it follows from Proposition \ref{p1}.

Suppose that $0\leq r\leq 1$. The following inequality implies that
$\|f(\cdot)\|$ is quasi-convex on $\sigma((0,\infty))$.
\begin{align*}
\left\|\big((1-t)A+tB\big)^r\right\|&=\left\|(1-t)A+tB\right\|^r\leq \big((1-t)\|A\|+t\|B\|\big)^r\\
&\leq\left(\max\{\|A\|, \|B\|\}\right)^r=\max\{\|A^r\|, \|B^r\|\}.
\end{align*}
In this case, we also have
\begin{align*}
\left\|\big((1-t)A+tB\big)^r\right\|&\leq \big((1-t)\|A\|+t\|B\|\big)^r\leq(1-t)^r\|A\|^r+t^r\|B\|^r\\
&=(1-t)^r\|A^r\|+t^r\|B^r\|.
\end{align*}
Therefore $\|f(\cdot)\|$ is r-convex on $\sigma((0,\infty))$.
\item[(ii)]
For the real function $f(t)=t^2-1$ on $\mathbb{R}$, the function $\|f(\cdot)\|$ is not quasi-convex, therefore is not
convex. Because
\begin{align*}
1= \|f(0)\|=\left\|f\left(\frac{-1_H+1_H}{2}\right)\right\|\nleqslant
 \max\{\|f(-1_H)\|,\|f(1_H)\|\}=0.
\end{align*}
\end{enumerate}
\end{example}

\begin{theorem}\label{t.1}
Let $f$ be an operator monotone function on $(0,\infty)$ and $n$ a positive integer. Then, $\|f^{(n)}(\cdot)\|$ is quasi-convex, i.e.,
\begin{align}\label{2.1}
\|f^{(n)}((1-\nu)A+\nu B)\|\leq \max\{\|f^{(n)}(A)\|, \|f^{(n)}(B)\|\},
\end{align}
for all positive definite operators $A, B$ and $0\leq\nu\leq 1$.
\end{theorem}

\begin{proof}
It is known that every operator monotone function
$f$ on $(0,\infty)$ has a special integral representation as follows:
\begin{align}\label{2.2}
f(t)=\alpha+\beta
t+\int_0^\infty\left(\frac{\lambda}{\lambda^2+1}-\frac{1}{\lambda+t}\right)d\mu(\lambda),
\end{align}
where $\alpha, \beta$ are real numbers, $\beta \geq 0$, and $\mu$ is
a positive measure on $(0,\infty)$ \cite[(V.49)]{bha}, such that
\begin{equation*}
\int_0^\infty\frac{1}{\lambda^2+1}d\mu(\lambda)<\infty.
\end{equation*}
The integral representation \eqref{2.2} implies that every operator monotone function on $(0,\infty)$
is infinitely differentiable. Therefore

\begin{align}\label{2.3}
f(A)=\alpha 1_H+\beta
A+\int_0^\infty\left[\frac{\lambda}{\lambda^2+1}1_H-\left(\lambda+A\right)^{-1}\right]d\mu(\lambda).
\end{align}
From the integral representation we also have
\begin{align}\label{2.4}
f'(t)=\beta+\int_0^\infty
\frac{1}{\left(\lambda+t\right)^2}d\mu(\lambda).
\end{align}
Hence
 \begin{align*}
 \left\|f'(A)\right\|=\left\|\beta 1_H+\int_0^\infty\left(\lambda+A\right)^{-2}d\mu(\lambda)\right\|.
 \end{align*}

 Let $a_0=\inf\{\langle Ax,x\rangle :\|x\|=1\}$ and  $b_0=\inf\{\langle Bx,x\rangle :\|x\|=1\}$. Then since $\beta\geq0$, we have
 \begin{align*}
\left\|f'(A)\right\|=\beta+\int_0^\infty\left(\lambda+a_0\right)^{-2}d\mu\left(\lambda\right)
=\beta+\int_0^\infty\left\|(\lambda+A)^{-1}\right\|^2d\mu\left(\lambda\right),
 \end{align*}
 and similarly,
 \begin{align*}
\left\|f'(B)\right\|=\beta+\int_0^\infty\left(\lambda+b_0\right)^{-2}d\mu\left(\lambda\right)
=\beta+\int_0^\infty\left\|(\lambda+B)^{-1}\right\|^2d\mu\left(\lambda\right).
 \end{align*}

 Suppose that $ a_0\leq b_0$, then $a_0\leq(1-\nu)a_0+\nu b_0\leq\inf\{\langle [(1-\nu)A+\nu B ]x,x\rangle :\|x\|=1\}$.
 This implies that $\|f'(A)\|\geq \|f'((1-\nu)A+\nu B)\|$, and proves the desired inequality (\ref{2.1}), for $n=1$.

 Analogously for $n\geq2$, we have
 \begin{align*}
 \left\|f^{(n)}(A)\right\|&=\left\|(-1)^{n+1}n!\int_0^\infty\left(\lambda+A\right)^{-n-1}d\mu(\lambda)\right\|\\
&=n!\int_0^\infty\left\|(\lambda+A)^{-1}\right\|^{n+1}d\mu(\lambda)
 \end{align*}

 The same argument will show
that in this case again the inequality (\ref{2.1}) holds.
\end{proof}
The following theorem generalises Theorem X.3.4 in \cite{bha}.
\begin{theorem}\label{t2}
Let $f$ be an operator monotone function on $(0,\infty)$ and  $|||.|||$ be a unitarily invariant norm on $B(H)$.
Let $A$ be any positive definite operator, then
 $$|||D^n f(A)|||\leq\|f^{(n)}(A)\|,$$ for all positive integers $n$.
\end{theorem}

\begin{proof}
From (\ref{2.3}) we have
\begin{equation*}
Df(A)(B)=\beta B+\int_0^\infty (\lambda +A)^{-1}B(\lambda +A)^{-1}d\mu(\lambda).
\end{equation*}
Since $|||(\lambda +A)^{-1}B(\lambda +A)^{-1}|||\leq \|(\lambda +A)^{-1}\|~|||B|||~\|(\lambda +A)^{-1}\|$,
\cite[(IV.40)]{bha}.
Hence,
\begin{equation*}
|||Df(A)(B)|||\leq\beta +\int_0^\infty \|(\lambda +A)^{-1}\|^2d\mu(\lambda)=\|f'(A)\|.
\end{equation*}

For $n\geq2$, we have
\begin{align*}
&D^n f(A)(B_1,B_2,...B_n)=(-1)^{n+1}\int_0^\infty \left[\sum_\sigma (\lambda +A)^{-1}B_{\sigma(1)}(\lambda +A)^{-1}\right.\\
&\dots~(\lambda +A)^{-1}B_{\sigma(n)}(\lambda +A)^{-1}\Big]d\mu(\lambda).
\end{align*}

Using (IV.40) in \cite{bha}, we obtain
\begin{align*}
|||D^n f(A)|||\leq n!\int_0^\infty \|(\lambda +A)^{-1}\|^{n+1}d\mu(\lambda)=\left\|f^{(n)}(A)\right\|.
\end{align*}

\end{proof}

Since the operator norm is a unitarily invariant norm, Theorem \ref{t2} implies that $f\in \mathcal{D}^{(n)}$.

\subsection{Hermite-Hadamard type inequalities}

We continue this section with the following technical lemma.

\begin{lemma}\cite[Lemma 1]{gha}\label{l1}
Let $I\subseteq\mathbb{R}$ be an open interval and $f : I
\rightarrow\mathbb{R}$ be a twice differentiable function on $I$
whose second derivative $f{''}$ is continuous on $I$ and
$\nu\in[0,1]$. If $A$ and $B$ are self adjoint operators with spectra in $I$, then
\begin{multline}\label{2.5}
\int_0^1\left(t-\nu\right)Df\left((1-t)A+tB\right)(B-A)dt\\
=\nu f(A)+(1-\nu)f(B)-\int_0^1f\left((1-t)A+tB\right)dt.
\end{multline}
\end{lemma}

\begin{proposition}\label{p2}
Let $f $ be a real function on $(0,\infty)$.
If $f\in C^2((0,\infty))\cap\mathcal{D}^{(1)}$ and $\|f'(\cdot)\|$ is
convex on $\sigma\left((0,\infty)\right)$, then
\begin{multline}\label{2.6}
\left\|\nu f(A)+(1-\nu)f(B)-\int_0^1f\left((1-t)A+tB\right)dt\right\|\\
\leq\frac{1}{6}\left[2(1-\nu)^3-3(1-\nu)+2\right]\|f'(A)\|\|B-A\|\\
+\frac{1}{6}\left(2\nu^3-3\nu+2\right)\|f'(B)\|\|B-A\|,
\end{multline}
for every positive definite operators $A,B$ and $0\leq\nu\leq1$.
\end{proposition}

\begin{proof}
Using Lemma \ref{l1}, we have
\begin{align*}
&\left\|\nu f(A)+(1-\nu)f(B)-\int_0^1f\left((1-t)A+tB\right)dt\right\|\\
&=\left\|\int_0^1(t-\nu)Df\left((1-t)A+tB\right)(B-A)dt\right\|\\
&\leq\|B-A\|\int_0^1|t-\nu|\left\|Df\left((1-t)A+tB\right)\right\|dt\\
&=\|B-A\|\int_0^1|t-\nu|\left\|f'\left((1-t)A+tB\right)\right\|dt\\
&\leq
\|B-A\|\int_0^1\Big(|t-\nu|(1-t)\|f'(A)\|+t\|f'(B)\|\Big)dt,
\end{align*}
since  $\|f'(\cdot)\|$ is convex. Now, by the following
equalities  we get the desired inequality \eqref{2.6}.
\begin{align*}
&\int_0^1|t-\nu|t dt=\frac{1}{6}\left(2\nu^3-3\nu+2\right),\\
&\int_0^1|t-\nu|(1-t) dt=\frac{1}{6}\left(2(1-\nu)^3-3(1-\nu)+2\right).
\end{align*}
\end{proof}

We give new inequalities as to the Hermite-Hadamard inequality in the following remark.
By the same argument used in the proof of Proposition \ref{p2} and by the following equalities
\begin{align*}
&\int_0^1|t-\nu| dt=\nu^2-\nu+\frac{1}{2}\\
&\int_0^1|t-\nu|t^s dt=\left(\frac{1}{s+2}-\frac{\nu}{s+1}+\frac{2\nu^{s+2}}{(s+1)(s+2)}\right)\\
&\int_0^1|t-\nu|(1-t)^s dt=\left(\frac{1}{s+2}-\frac{1-\nu}{s+1}+\frac{2(1-\nu)^{s+2}}{(s+1)(s+2)}\right).
\end{align*}

\begin{proposition}\label{p3}
Let $f,g$ be two operator monotone functions on $(0,\infty)$. Then, for every unitarily invariant norm $|||.|||$ and
every positive definite operators $A,B$,
\begin{align}\label{2.60}
&\left|\left|\left|\nu f(A)g(A)+(1-\nu)f(B)g(B)-\int_0^1(fg)\left((1-t)A+tB\right)dt\right|\right|\right|\notag\\
&\leq\left(\nu^2-\nu+\frac{1}{2}\right)|||B-A|||\Big[\max\left\{\|f'(A)\|,\|f'(B)\|\right\}\times\max\left\{\|g(A)\|,\|g(B)\|\right\}\notag\\
&+\max\left\{\|f(A)\|,\|f(B)\|\right\}\times \max\left\{\|g'(A)\|,\|g'(B)\|\right\}\Big].
\end{align}
\end{proposition}

\begin{proof}
Using Lemma \ref{l1}, we get
\begin{align}\label{2.61}
&\left|\left|\left|\nu f(A)g(A)+(1-\nu)f(B)g(B)-\int_0^1(fg)\left((1-t)A+tB\right)dt\right|\right|\right|\notag\\
&=\left|\left|\left|\int_0^1(t-\nu)D(fg)\left((1-t)A+tB\right)dt\right|\right|\right|\notag\\
&=\Big|\Big|\Big|\int_0^1(t-\nu)\Big[Df\left((1-t)A+tB\right)(B-A)g((1-t)A+tB)\notag\\
&+f((1-t)A+tB)Dg((1-t)A+tB)(B-A)dt\Big]\Big|\Big|\Big|\notag\\
&\leq \int_0^1|t-\nu|\Big[|||Df\left((1-t)A+tB\right)(B-A)g((1-t)A+tB)|||\notag\\
&+|||f((1-t)A+tB)Dg((1-t)A+tB)(B-A)|||dt\Big].
\end{align}
By inequality (IV.40) in \cite{bha}, we have
$$|||XY|||\leq ||X||~|||Y||| \text{ and }|||XY|||\leq |||X|||~||Y||~(X,Y\in B(H).$$
Therefore,
\begin{align}\label{2.62}
&\int_0^1|t-\nu|\Big[|||Df\left((1-t)A+tB\right)(B-A)g((1-t)A+tB)|||\notag\\
&+|||f((1-t)A+tB)Dg((1-t)A+tB)(B-A)|||dt\Big]\notag\\&\leq \int_0^1|t-\nu|\Big[|||Df\left((1-t)A+tB\right)(B-A)|||~||g((1-t)A+tB)||\notag\\
&+||f((1-t)A+tB)||~|||Dg((1-t)A+tB)(B-A)|||dt\Big]\notag\\
&\leq |||B-A|||\int_0^1|t-\nu|\Big[|||Df\left((1-t)A+tB\right)|||~||g((1-t)A+tB)||\notag\\
&+||f((1-t)A+tB)||~|||Dg((1-t)A+tB)|||dt\Big]\notag\\
&\leq |||B-A|||\int_0^1|t-\nu|\Big[||f'\left((1-t)A+tB\right)||~||g((1-t)A+tB)||\notag\\
&+||f((1-t)A+tB)||~||g'((1-t)A+tB)||dt\Big].
\end{align}
By Proposition \ref{p1} and Theorem \ref{t.1}, the functions $||f(\cdot)||, ||f'(\cdot)||,||g(\cdot)||,||g'(\cdot)||$
are quasi-convex. Therefore from relations \eqref{2.61} and \eqref{2.62}, we obtain inequality \eqref{2.60}.
\end{proof}

\begin{corollary}\label{p3}
Let $f,g$ be two operator monotone functions on $(0,\infty)$. Then, for every unitarily invariant norm $|||.|||$ and
every positive definite operators $A,B$,
\begin{align*}
&\left|\left|\left|f(A)g(A)-f(A)g(B)\right|\right|\right|\notag\\
&\leq|||B-A|||\Big[\max\left\{\|f'(A)\|,\|f'(B)\|\right\}\times\max\left\{\|g(A)\|,\|g(B)\|\right\}\notag\\
&+\max\left\{\|f(A)\|,\|f(B)\|\right\}\times \max\left\{\|g'(A)\|,\|g'(B)\|\right\}\Big].
\end{align*}
\end{corollary}

From inequality \eqref{2.60} for $g(t)=1$, we get the following result.

\begin{corollary}
Let $f$ be an operator monotone function on $(0,\infty)$. Then, for every unitarily invariant norm $|||.|||$ and
every positive definite operators $A,B$
\begin{align}\label{2.9}
&\left|\left|\left|\nu f(A)+(1-\nu)f(B)-\int_0^1f\left((1-t)A+tB\right)dt\right|\right|\right|\notag\\
&\leq\left(\nu^2-\nu+\frac{1}{2}\right)|||B-A|||\max\left\{\|f'(A)\|,\|f'(B)\|\right\}.
\end{align}
\end{corollary}

\begin{remark}
\begin{enumerate}
Let $f $ be a real function on $(0,\infty)$,
$f\in C^2((0,\infty))\cap\mathcal{D}^{(1)}$, and $0\leq\nu\leq1$. Put
\begin{equation*}
X:=\nu f(A)+(1-\nu)f(B)-\int_0^1f\left((1-t)A+tB\right)dt.
\end{equation*}
\item[(i)] If $\|f'(\cdot)\|$ is
quasi-convex on $\sigma((0,\infty))$, then for every positive definite operators $A,B$
\begin{align}\label{2.7}
\|X\|\leq\left(\nu^2-\nu+\frac{1}{2}\right)\|B-A\|\max\left\{\|f'(A)\|,\|f'(B)\|\right\}.
\end{align}
\item[(ii)] If $\|f'(\cdot)\|$ is $s$-convex on $\sigma((0,\infty))$, then for every positive definite operators $A,B$
\begin{align}\label{2.8}
\|X\|&\leq\left(\frac{1}{s+2}-\frac{1-\nu}{s+1}+\frac{2(1-\nu)^{s+2}}{(s+1)(s+2)}\right)\|f'(A)\|\|B-A\|\notag\\
&+\left(\frac{1}{s+2}-\frac{\nu}{s+1}+\frac{2\nu^{s+2}}{(s+1)(s+2)}\right)\|f'(B)\|\|B-A\|.
\end{align}
\end{enumerate}
\end{remark}

\subsection{norm inequalities for Matrices}

Let $A , B, X\in M_n(\mathbb{C})$ such that $A$ and $B$ be positive definite and $0\leq\nu\leq1$.
A difference version of the Heinz inequality
\begin{equation}\label{4.0.0}
|||A^\nu XB^{1-\nu}-A^{1-\nu}XB^\nu|||\leq |2\nu-1|~|||AX-XB|||
\end{equation}
was proved by Bhatia and Davis in \cite{bah5}.

Kapil, et.al.,\cite{kap1} proved that if $0<r\leq1$. Then
\begin{equation}\label{4.0}
|||A^rX-XB^r|||\leq r\max\{||A^{r-1}||, ||B^{r-1}||\} |||AX-XB|||.
\end{equation}
They also proved that if $\alpha\geq1$, and $\frac{1-\alpha}{2}\leq\nu\leq\frac{1+\alpha}{2}$, then
\begin{align}\label{4.1}
\alpha |||A^\nu XB^{1-\nu}&-A^{1-\nu}XB^\nu|||\notag\\
&\leq |2\nu-1|\max\{||A^{1-\alpha}||, ||B^{1-\alpha}||\} |||A^\alpha X-XB^\alpha|||.
\end{align}

The following theorem is a generalization of \eqref{4.0}.
\begin{theorem}\label{t3}
Let $A , B, X\in M_n(\mathbb{C})$ such that $A$ and $B$ be positive definite and $f,g$ be two operator monotone functions on $(0,\infty)$. Then
\begin{align}\label{4.2}
&|||f(A)g(A)X-Xf(B)g(B)|||\notag\\
&\leq |||AX-XB|||\Big[\max\{||f'(A)||, ||f'(B)||\}\times \max\{||g(A)||, ||g(B)||\}\notag\\
&+\max\{||f(A)||, ||f(B)||\}\times \max\{||g'(A)||, ||g'(B)||\} \Big].
\end{align}
\end{theorem}

\begin{proof}
It suffices to prove the required inequality in the special case which
$A = B$ and $A$ is diagonal. Then the general case follows by replacing $A$ with
$\begin{bmatrix}
  A & 0 \\
  0 & B \\
\end{bmatrix}$ and $X$ with
$\begin{bmatrix}
  0 & X \\
  0 & 0 \\
\end{bmatrix}.$
Therefore let $A=diag(\lambda_i)>0$. Then $f(A)g(A)X-Xf(A)g(A)=Y\circ (AX-XA)$ where $Y=(fg)^{[1]}(A)$, i.e.,

\begin{equation*}
y_{ij}= \\
\begin{cases}
\frac{f(\lambda_i)g(\lambda_i)-f(\lambda_j)g(\lambda_j)}{\lambda_i-\lambda_j}, &\lambda_i\neq \lambda_j\\
f'(\lambda_i)g(\lambda_i)+f(\lambda_i)g'(\lambda_i),&\lambda_i=\lambda_j.
\end{cases}
\end{equation*}
Or, $Y=Pg(A)+f(A)Q$, where $P=f^{[1]}(A)$, $Q=g^{[1]}(A)$ and
\begin{equation*}
p_{ij}= \\
\begin{cases}
\frac{f(\lambda_i)-f(\lambda_j)}{\lambda_i-\lambda_j}, &\lambda_i\neq \lambda_j\\
f'(\lambda_i),&\lambda_i=\lambda_j,
\end{cases}
\end{equation*}
and
\begin{equation*}
q_{ij}= \\
\begin{cases}
\frac{g(\lambda_i)-g(\lambda_j)}{\lambda_i-\lambda_j}, &\lambda_i\neq \lambda_j\\
g'(\lambda_i),&\lambda_i=\lambda_j.
\end{cases}
\end{equation*}
By \cite[Theorem V.3.4]{bha}, $f^{[1]}(A)\geq0$ and $g^{[1]}(A)\geq0$.
Consequently
\begin{align*}
|||f(A)g(A)X&-Xf(A)g(A)|||=|||Y\circ (AX-XA)|||\\
&=|||(Pg(A+f(A)Q)\circ (AX-XA)|||\\
&\leq |||Pg(A)\circ (AX-XA)|||+|||f(A)Q\circ (AX-XA)|||\\
&\leq\max p_{ii}g(\lambda_i)~|||AX-XA|||+\max f(\lambda_i)q_{ii}~|||AX-XA|||\\
&\leq\big[\|f'(A)\|\|g(A)\|+\|f(A)\|\|g'(A)\|\big]~|||AX-XA|||.
\end{align*}
\end{proof}
From \eqref{4.2} for $g(t)=1$, we get the following result.

\begin{theorem}\cite[Theorem 1]{gha2},\label{t4}
Let $A , B, X\in M_n(\mathbb{C})$ such that $A$ and $B$ be positive definite and $f$ be an operator monotone function on $(0,\infty)$. Then
\begin{equation}\label{4.3}
|||f(A)X-Xf(B)|||\leq \max\{||f'(A)||, ||f'(B)||\} |||AX-XB|||.
\end{equation}
\end{theorem}

\begin{example}
(i) For the function $f(t)=t^r,~0<r\leq1$,
\begin{align*}
\left|\left|\left|A^rX-XB^r\right|\right|\right|&\leq r\max\{\|A^{r-1}\|, \|B^{r-1}\|\}|||AX-XB|||\\
&=r\left(\max\{\|A^{-1}\|, \|B^{-1}\|\}\right)^{1-r}|||AX-XB|||.
\end{align*}
(ii) For the function $f(t)=\log t$ on $(0,\infty)$,
\begin{equation*}
\left|\left|\left|\log(A)X-X\log(B)\right|\right|\right|\leq
\left(\max\{\|A^{-1}\|, \|B^{-1}\|\}\right)|||AX-XB|||.
\end{equation*}
(iii) For the functions $f(t)=t^r,~0<r\leq1$ and $g(t)=\log t$ on $(0,\infty)$. Inequality \eqref{4.2} implies that
\begin{align*}
|||A^r\log(A)X&-XB^r\log(B)||| \\
&\leq\Big[r\left(\max\{\|A^{-1}\|, \|B^{-1}\|\}\right)^{1-r}\times \log\left(\max\{\|A\|, \|B\|\}\right)\\
&+\left(\max\{\|A\|, \|B\|\}\right)^{r}\times\max\{\|A^{-1}\|, \|B^{-1}\|\}\Big]|||AX-XB|||.
\end{align*}

\end{example}

\begin{remark}
Let $\alpha\geq1$ and $0\leq\nu\leq1$. From inequality \eqref{4.2} for $A^\alpha, B^\alpha$ and $f(t)=t^\frac{1}{\alpha}$, we get
\begin{equation}\label{4.2.0}
|||AX-XB|||\leq \frac{1}{\alpha}\max\{||A^{1-\alpha}||, ||B^{1-\alpha}||\} |||A^\alpha X-XB^\alpha |||.
\end{equation}
On combining \eqref{4.0.0}, and \eqref{4.2.0}, we obtain \eqref{4.1}.
\end{remark}

\section{\bf Applications}

As an important application of the results in this paper, we find
bounds for $\|f(B)-f(A)\|$ in terms of $\|B-A\|$, which is one of
the central problems in perturbation theory.

The following companions of the inequalities $(X,43)-(X,46)$ in \cite{bha} for operator monotone functions holds.

\begin{corollary}\label{c1}
\begin{enumerate}\item[(i)] Let $f $ be a real function on $(0,\infty)$ and
$f\in C^2((0,\infty))\cap\mathcal{D}^{(1)}$, then for every positive definite operators $A,B$
\begin{equation*}
\left\|f(B)-f(A)\right\|\leq \\
\begin{cases}
\frac{1}{2}\left(\|f'(A)\|+\|f'(B)\|\right)\|B-A\|, &\text{$\|f'(\cdot)\|$ is convex}\\
\max\{\|f'(A)\|,\|f'(B)\|\}\|B-A\|,&\text{$\|f'(\cdot)\|$ quasi-convex}\\
\frac{1}{s+1}(\|f'(A)\|+\|f'(B)\|)~\|B-A\|, &\text{$\|f'(\cdot)\|$ is $s$-convex}
\end{cases}
\end{equation*}
\item[(ii)] Let $f$ be an operator monotone function on $(0,\infty)$. Then, for every unitarily invariant norm $|||.|||$ and
every positive definite operators $A,B$
\begin{equation}\label{3.1}
\left|\left|\left|f(B)-f(A)\right|\right|\right|\leq
\max\{\|f'(A)\|,\|f'(B)\|\}\left|\left|\left|B-A\right|\right|\right|.
\end{equation}
\end{enumerate}
\end{corollary}

\begin{proof}
(i) Let $\|f'(\cdot)\|$ be convex function; using \eqref{2.6} for $\nu=1$ and $\nu=0$, we get the following inequalities,

\begin{align*}
\left\|f(A)-\int_0^1f\left((1-t)A+tB\right)dt\right\|\leq\left(\frac{1}{3}\|f'(A)\|+\frac{1}{6}\|f'(B)\|\right)\|B-A\|,
\end{align*}
and
\begin{align*}
\left\|f(B)-\int_0^1f\left((1-t)A+tB\right)dt\right\|\leq\left(\frac{1}{6}\|f'(A)\|+\frac{1}{3}\|f'(B)\|\right)\|B-A\|.
\end{align*}
If $\|f'(\cdot)\|$ is quasi-convex or $s$-convex, then from inequalities in \eqref{2.7} or \eqref{2.8}
for $\nu=1$ and $\nu=0$, we obtain the desired inequalities in part (i).

(ii) Utilizing  \eqref{2.9} for $\nu=1$ and $\nu=0$, we obtain the following inequalities, which are interesting in their own right.

\begin{align*}
\left|\left|\left|f(A)-\int_0^1f\left((1-t)A+tB\right)dt\right|\right|\right|
\leq\frac{1}{2}|||B-A|||\max\left\{\|f'(A)\|,\|f'(B)\|\right\},
\end{align*}
and
\begin{align*}
\left|\left|\left|f(B)-\int_0^1f\left((1-t)A+tB\right)dt\right|\right|\right|
\leq\frac{1}{2}|||B-A|||\max\left\{\|f'(A)\|,\|f'(B)\|\right\}.
\end{align*}
Hence we obtain the desired inequality \eqref{3.1} and the proof is completed.
\end{proof}

From the Corollary \ref{c1}, we may state the following example:

\begin{example}
(i) Let $r\in \mathbb{R}$, we know that the real function $f(t)=t^{r}$ on $(0, \infty)$ is in
$C^2((0. \infty))\cap\mathcal{D}^{(1)}$, if $r\notin (1,2)$; see \cite{bha4}.
Since $\|f'(\cdot)\|$ is convex for $r\notin (1,2)$, and is $(r-1)$-convex for
$1< r < 2$, we obtain the following inequalities
\begin{equation*}
\left\|B^r-A^r\right\|\leq \\
\begin{cases}
\left(\|B^{r-1}\|+\|A^{r-1}\|\right)\|B-A\| & r\in (1,2)\\
\frac{r}{2}(\|B^{r-1}\|+\|A^{r-1}\|)~\|B-A\| & r\notin (1,2).
\end{cases}
\end{equation*}

(ii) The function $f(t)=t^r,~0\leq r\leq1$, is operator monotone, therefore
\begin{equation*}
\left|\left|\left|B^r-A^r\right|\right|\right|\leq
r\left(\max\{\|A^{-1}\|, \|B^{-1}\|\}\right)^{1-r}|||B-A|||.
\end{equation*}

(iii) The function $f(t)=\log t$ on $(0,\infty)$, is operator monotone, so
\begin{equation*}
\left|\left|\left|\log(B)-\log(A)\right|\right|\right|\leq
\left(\max\{\|A^{-1}\|, \|B^{-1}\|\}\right)|||B-A|||.
\end{equation*}
(iv) The functions $f(t)=t^r,~0\leq r\leq1$ and $g(t)=\log t$ on $(0,\infty)$, are operator monotone.
From \eqref{2.60}, we obtain
\begin{align*}
&|||B^r\log B-A^r\log A|||\\
&\leq |||B-A|||\Big[r\left(\max\{\|A^{-1}\|, \|B^{-1}\|\}\right)^{1-r}\times\log(\max\{\|A\|, \|B\|\})\\
&+(\max\{\|A\|, \|B\|\})^r\times\max\{\|A^{-1}\|, \|B^{-1}\|\}\Big].
\end{align*}

\end{example}

\begin{remark}
Let $f$ be an operator monotone function on $(0,\infty)$ and let $A, B$ be two positive definite operators such that $ A\geq a1_H$
and $B \geq a1_H$ for the positive number $a$. By \cite[Theorem X.3.8]{bha}, for every unitarily invariant norm the following inequality holds
\begin{equation}\label{3.2}
\left|\left|\left|f(B)-f(A)\right|\right|\right|\leq
f'(a)\left|\left|\left|B-A\right|\right|\right|.
\end{equation}

From inequality \eqref{2.4}, we have

\begin{align*}
 f'(A)&=\beta 1_H+\int_0^\infty\left(\lambda+A\right)^{-2}d\mu(\lambda)\\
 &\leq \beta 1_H+\left[\int_0^\infty\left(\lambda+a\right)^{-2}d\mu(\lambda)\right]I=f'(a)1_H.
 \end{align*}

This implies that $\max\{\|f'(A)\|,\|f'(B)\|\}\leq f'(a)$; therefore inequality \eqref{3.1} is a refinement of \eqref{3.2}.
This refinement can be sharp because for the function $f(t)=t^r,~0<r<1$ and $A=B=b1_H$ with $b>a$, we have
$$\max\{\|f'(A)\|, \|f'(B)\|\}=r\max\{\|A^{r-1}\|, \|B^{r-1}\|\}=rb^{r-1}<ra^{r-1}=f'(a).$$
\end{remark}

 In the following Theorems, we give two error estimates of the Simpson's rules for operator monotone functions as integrand.

\begin{theorem}\label{t.5}
Let $f$ be an operator monotone function on $(0,\infty)$. Then, for every unitarily invariant norm $|||.|||$ and every positive definite operators $A,B$,
\begin{multline}\label{3.3}
\left|\left|\left|\frac{f(A)+4f\left(\frac{A+B}{2}\right)+f(B)}{6}-\int_0^1f\left((1-t)A+tB\right)dt\right|\right|\right|\\
\leq \frac{5}{32}\left|\left|\left|B-A\right|\right|\right|\max\left\{\|f'(A)\|,\|f'(B)\|\right\}.
\end{multline}

\end{theorem}

\begin{proof}
It is easy to show that
\begin{align*}
\int_0^1 f((1-t)A+tB)dt&=\frac{1}{2}\int_0^1 f\left((1-t)A+t\frac{A+B}{2}\right)dt\\
&+\frac{1}{2}\int_0^1 f\left((1-t)\frac{A+B}{2}+tB\right)dt.
\end{align*}
Regarding to $\left\|f'\left(\frac{A+B}{2}\right)\right\|\leq\max\left\{\|f'(A)\|,\|f'(B)\|\right\}$,
and inequality \eqref{2.9} for $\nu=\frac{1}{3}$, we have
\begin{multline}\label{3.4}
\left|\left|\left|\frac{f(A)+2f\left(\frac{A+B}{2}\right)}{3}-\int_0^1f\left((1-t)A+t\frac{A+B}{2}\right)dt\right|\right|\right|\\
\leq\frac{5}{32}\left|\left|\left|B-A\right|\right|\right|\max\left\{\|f'(A)\|,\left\|f'\left(\frac{A+B}{2}\right)\right\|\right\}\\
\leq \frac{5}{32}\left|\left|\left|B-A\right|\right|\right|\max\left\{\|f'(A)\|,\|f'(B)\|\right\}.
\end{multline}
Again, using inequality \eqref{2.9} for $\nu=\frac{2}{3}$, we have
\begin{multline}\label{3.5}
\left|\left|\left|\frac{2f\left(\frac{A+B}{2}\right)+f(B)}{3}-\int_0^1f\left((1-t)\frac{A+B}{2}+tB\right)dt\right|\right|\right|\\
\leq\frac{5}{32}\left|\left|\left|B-A\right|\right|\right|\max\left\{\left\|f'\left(\frac{A+B}{2}\right)\right\|, \|f'(B)\|\right\}\\
\leq \frac{5}{32}\left|\left|\left|B-A\right|\right|\right|\max\left\{\|f'(A)\|,\|f'(B)\|\right\}.
\end{multline}
From relations \eqref{3.4}, \eqref{3.5}, and the following equation, we obtain the desired inequality \eqref{3.3}.
\begin{align*}
\frac{1}{6}\left(f(A)+4f\left(\frac{A+B}{2}\right)+f(B)\right)&=\frac{1}{2}\left(\frac{f(A)+2f\left(\frac{A+B}{2}\right)}{3}\right)\\
&+\frac{1}{2}\left(\frac{2f\left(\frac{A+B}{2}\right)+f(B)}{3}\right),
\end{align*}
Hence the proof is completed.
\end{proof}

A simple calculation show that
\begin{align*}
&\int_0^1 f((1-t)A+tB)dt=\frac{1}{3}\int_0^1 f\left((1-t)A+t\frac{2A+B}{3}\right)dt\\
&+\frac{1}{3}\int_0^1 f\left((1-t)\frac{2A+B}{3}+t\frac{A+2B}{3}\right)dt+\frac{1}{3}\int_0^1 f\left((1-t)\frac{A+2B}{3}+tB\right)dt,
\end{align*}
and
\begin{multline*}
\frac{1}{8}\left(f(A)+3f\left(\frac{2A+B}{3}\right)+3f\left(\frac{A+2B}{3}\right)+f(B)\right)\\
=\frac{1}{3}\left(\frac{3}{8}f(A)+\frac{5}{8}f\left(\frac{2A+B}{3}\right)\right)
+\frac{1}{3}\left(\frac{1}{2}f\left(\frac{2A+B}{3}\right)+\frac{1}{2}f\left(\frac{A+2B}{3}\right)\right)\\
+\frac{1}{3}\left(\frac{5}{8}f\left(\frac{A+2B}{3}\right)+\frac{3}{8}f\left(B\right)\right).\\
\end{multline*}
Therefore, we deduce the following result.

\begin{theorem}\label{t.6}
Let $f$ be an operator monotone function on $(0,\infty)$. Then, for every unitarily invariant norm $|||.|||$ and every positive definite operators $A,B$,
\begin{multline*}
\left|\left|\left|\frac{f(A)+3f\left(\frac{2A+B}{3}\right)+3f\left(\frac{A+2B}{3}\right)+f(B)}{8}-\int_0^1f\left((1-t)A+tB\right)dt\right|\right|\right|\\
\leq \frac{25}{288}\left|\left|\left|B-A\right|\right|\right|\max\left\{\|f'(A)\|,\|f'(B)\|\right\}.
\end{multline*}

\end{theorem}

\end{document}